\documentclass[10pt]{amsart}

\usepackage{hyperref}
\hypersetup{nesting=true,debug=true,naturalnames=true}
\usepackage{graphicx,amssymb,upref,units,}

\usepackage{tikz}

\let\<\langle
\let\>\rangle

\let\uml\"

\usepackage{amsmath,amssymb,,xy,mathrsfs,amsthm,xcolor,amsxtra,graphicx,verbatim}
\usepackage{hyperref}
\hypersetup{colorlinks=true,urlcolor=magenta,citecolor=magenta,linkcolor=blue}

\usepackage{colonequals}
\usepackage{multirow}

\usepackage[mathcal]{euscript}
\newcommand{\Gal}{\operatorname{Gal}}

\newcommand{\F}{\mathbf{F}}

\newcommand{\Q}{\mathbf{Q}}

\newcommand{\ddef}{\colonequals}
\newcommand{\disc}{\operatorname{disc}}

\newcommand{\lcm}{\operatorname{lcm}}

\newcommand{\Res}{\operatorname{Res}}

\theoremstyle{plain}
\newtheorem{thm}[equation]{Theorem}

\newtheorem{cor}[equation]{Corollary}
\newtheorem{prop}[equation]{Proposition}

\theoremstyle{remark}
\newtheorem{rmk}[equation]{Remark}
\newtheorem{remark}[equation]{Remark}
\newtheorem{exm}[equation]{Example}

 \input xy
 \xyoption{all}

\begin{document}

\title{On the discriminants of truncated logarithmic polynomials}

\author{John Cullinan \and Rylan Gajek-Leonard}

\address{Department of Mathematics, Bard College, Annandale-On-Hudson, NY 12504, USA}
\email{cullinan@bard.edu}
\urladdr{\url{http://faculty.bard.edu/cullinan/}}

\address{Department of Mathematics,
Union College,
Schenectady, NY 12308, USA}
\email{gajekler@union.edu}
\urladdr{\url{https://www.math.union.edu/~gajekler/}}


\begin{abstract}
We provide evidence for a conjecture of Yamamura that the truncated logarithmic polynomials 
\[
F_n(x) = 1 + x  + \frac{x^2}{2} + \cdots + \frac{x^n}{n}
\]
have Galois group $S_n$ for all $n \geq 1$.
\end{abstract}

\maketitle

\section{Introduction}

In \cite{ijnt}, the authors consider two families of polynomials:
\begin{align*}
F_n(x) &= 1 + x + \frac{x^2}{2} + \cdots + \frac{x^n}{n},\text{ and} \\
f_n(x) &= 1 - \frac{x^2}{2!} + \cdots + \frac{(-1)^nx^{2n}}{(2n)!}.
\end{align*}
These are the truncated Maclaurin polynomials of the functions $F(x) = 1 - \log (1-x)$ and $f(x) = \cos(x)$, respectively.  Motivated by Schur's pioneering result on the truncated exponential polynomials \cite{schur}, they work through the irreducibility and Galois groups of these polynomials.  In the case of $F_n$, they prove that $\Gal_\Q(F_n)$ contains $A_n$ for all $n$.  

However, their claim that $\Gal_\Q(F_n) \simeq S_n$ ``depending on the residue of the integer number $n$ modulo 4'' appears to be incorrect, as pointed out by Yamamura in his review of the paper for MathSciNet \cite{yam}.  In fact, Yamamura therein conjectures that $\Gal_\Q(F_n) \simeq S_n$ for all $n \geq 1$.  The purpose of this note is to address this conjecture and highlight the apparent difficulty of proving it in general.

By a Newton polygon argument in \cite{ijnt}, the $F_n$ are irreducible over $\Q$ with $\Gal_\Q(F_n) \supseteq A_n$ for all $n \geq 1$.  Thus, Yamamura's conjecture amounts to proving that the discriminant of $F_n$ is not a rational square once $n >1$.  Let us now set some notation.

The discriminant of a univariate polynomial $P(x)$ of degree $n$ and leading coefficient $a_n$ is given by
\begin{align} \label{res}
\disc P = (-1)^{\binom{n}{2}} a_n^{-1} \Res(P,P'). 
\end{align}
Computing the derivative
\[
F_n'(x) = 1 + x + x^2 + \cdots + x^{n-1},
\]
whose roots are the nontrivial $n$-th roots of unity,  the discriminant we seek is therefore
\begin{align} \label{disc}
\disc(F_n) =  (-1)^{\binom{n}{2}} n \prod_{\lbrace \theta : F_n'(\theta)=0 \rbrace} F_n(\theta).
\end{align}

Let $L_n = \lcm \lbrace 1,2,\dots,n \rbrace$. Then 
\[
F_n(x) = \frac{\widetilde{F}_n(x)}{L_n},
\]
where $\widetilde{F}_n(x)$ is a polynomial with integer coefficients.  The relations
\[
\theta^n = 1 \qquad \text{and} \qquad 1 + \theta + \theta^2 + \cdots + \theta^{n-1} = 0
\]
hold when $\theta$ is an $n$-th root of unity, which allow us to simplify $\widetilde{F}_n(\theta)$; we have
\begin{align} \label{simplified}
\widetilde{F}_{n}(\theta)=\sum_{k=0}^{n-2}a_k\theta^k \qquad \text{where}\qquad a_k=\begin{cases}L_n + \frac{L_n}{n} - \frac{L_n}{n-1} &\text{if $k=0$}\\
\frac{L_n}{k} - \frac{L_n}{n-1}&\text{if $k>0$}.
\end{cases}
\end{align}

Substituting (\ref{simplified}) into (\ref{disc}) and simplifying yields the following expression for the discriminant of $F_n(x)$:
\begin{align} \label{disc_form}
\disc(F_n) =  (-1)^{\binom{n}{2}} \frac{n}{L_n^{n-1}} \prod_{\lbrace \theta : F_n'(\theta)=0 \rbrace} \widetilde{F}_n(\theta).
\end{align}
Note that the product
\begin{align} \label{prod} 
\mathcal{P}_n \ddef \prod_{\lbrace \theta : F_n'(\theta)=0 \rbrace} \widetilde{F}_n(\theta)
\end{align}
is an integer.  As we will see below, the main difficulty in proving Yamamura's conjecture is understanding the prime factorization of $\mathcal{P}_n$.  We pause for an illustrative example.

\begin{exm} \label{n=9}
Let $n=9$.  Then we compute
\begin{align*}
\disc(F_9) &= \frac{9}{L_9^8} \prod_{\lbrace \theta : F_9'(\theta)=0 \rbrace} \widetilde{F}_9(\theta) = \left( \frac{3}{2^{12}3^{8}5^4 7^4} \right)^2 \mathcal{P}_9,
\end{align*}
hence $\disc(F_9)$ is a rational square if and only if $\mathcal{P}_9$ is.  In this case, we compute
\[
\mathcal{P}_9 = 1531\cdot  3137311 \cdot 113564970051005791,
\]
and these prime factors do not appear to be predictable \emph{a priori}.
\end{exm}

The above example  illustrates the main difficulty in proving $\disc(F_n) \not \in \Q^{\times 2}$ when $n$ is a square: the factor $\mathcal{P}_n$ is typically not divisible by primes dividing $n$, so we cannot use ``small'' primes to our advantage.  And, additionally, the prime divisors of $\mathcal{P}_n$ may be quite difficult or impossible to characterize.

In \cite{ijnt} the authors show that if $n \equiv 0,2,3 \pmod{4}$, then $\disc(F_n) \not \in \Q^{\times 2}$.  They then prove that if $n \equiv 1 \pmod{4}$ and $n$ is prime, then $\disc(F_n) \not \in \Q^{\times 2}$. In some cases, our results can be viewed as a companion to theirs: we were unable to easily verify some of their claims, and so we give independent proofs of these discriminant results, as well as some mild generalizations.  

\begin{thm} \label{mainthm1}
If $n \equiv 0,2,3 \pmod{4}$, or if $n \equiv 1 \pmod{4}$ and is the odd power of a prime, then $\disc(F_n) \not \in \Q^{\times 2}$.
\end{thm}

Regarding the prime divisors of $\mathcal{P}_n$, we can say the following as a generalization of Example \ref{n=9}.

\begin{thm} \label{sqthm}
If $n = p^{2e}$ for an odd prime $p$ and a positive integer $e$, then $\mathcal{P}_n$ is coprime to $p$.
\end{thm}

However, the main contribution of this paper is to give substantially new infinite families of integers $n \equiv 1 \pmod{4}$ for which $\disc (F_n)$ is not a rational square.

\begin{thm} \label{mainthm2} Let $m$ be a positive integer. 
For all but finitely many primes $q$, if $mq\equiv1\pmod{4}$ then $\disc(F_{mq}) \not \in \Q^{\times 2}$.
\end{thm}

\begin{rmk} For each integer $m$ one can explicitly compute the finite set of primes for which the above theorem does not apply. It is then a matter of computation to check that $\disc(F_n) \not \in \Q^{\times 2}$ for each of these exceptional primes. We give examples to demonstrate this in the final section of  this article. 
\end{rmk}

Before proving our results, we briefly survey the calculation of discriminants for number-theoretic purposes.

\section{Commentary on Discriminants, Resultants, and Arithmetic}

The systematic study of the algebraic properties of families of orthogonal polynomials appears to originate in two papers of Holt \cite{holt1,holt2}, in which he studies the irreducibility of certain Legendre polynomials.  Holt's methods were generalized significantly by Ille and Schur \cite{schur}, leading to extensive studies of the irreducibility and Galois theory of many families of hypergeometric polynomials (\emph{e.g.},~Jacobi, Laguerre, Chebyshev, Hermite).  Assuming irreducibility, there are Newton Polygon methods that will allow one to prove that the Galois group contains $A_n$.  Deciding whether or not the Galois group is all of $S_n$ is then a matter of determining when the discriminant is a rational square.  

The applications of orthogonal polynomials to arithmetic go beyond computational Galois theory.  In particular, if ${\rm ss}_p(t)\in \F_p[t]$ denotes the supersingular polynomial (whose roots are the supersingular $j$-invariants of elliptic curves over $\F_p$), then ${\rm ss}_p(t)$ is the reduction modulo $p$ of a certain Jacobi polynomial.  For more background and for other interesting lifts of ${\rm ss}_p(t)$ to $\Q$, see \cite{kz,morton}.  A conjecture of Mahburg and Ono \cite{mo} states that the ``Jacobi lifts'' are irreducible with maximal Galois group, and in \cite{cgl} we gave evidence for this conjecture using many of the techniques originally described by Schur.  

In the case of the classical families of orthogonal polynomials, the discriminants tend to have nice expressions.  A prototypical example is the case of the truncated exponential polynomials
\[
e_n(x) = \sum_{j=0}^n \frac{x^j}{j!},
\]
whose discriminants are given by $\disc (e_n) = \left(-1\right)^{\binom{n}{2}} \left(n! \right)^n$.  This formula is a key feature of Coleman's proof that $\Gal_\Q \left(e_n\right)$ always contains $A_n$ and equals $S_n$ if and only if $4 \nmid n$ \cite{coleman}.  We note that if 
\[
{\displaystyle L_{n}^{(\alpha )}(x)=\sum _{i=0}^{n}(-1)^{i}{n+\alpha  \choose n-i}{\frac {x^{i}}{i!}}}
\]
denotes the $n$th Generalized Laguerre Polynomial (GLP), then $e_n(x) = L_n^{(-1-n)}(x)$.

Ultimately, the reason that certain hypergeometric polynomials (including all of the ones mentioned above) have such tidy discriminant formulas is that they satisfy a Sturm-Liouville differential equation, relating the derivative of a member of the family to another member, and that the polynomials solve certain recurrence relations; for example, in the case of the GLP we have
\[
\frac{{\rm d}}{{\rm d}x} L_{n}^{(\alpha )}(x)= -L_{n-1}^{(\alpha+1)}(x).
\]
The recurrence relations, together with the differential equations, typically allow for explicit discriminant formulas, via (\ref{res}), which are amenable to arithmetic study.   

This is not the case with the truncated logarithmic polynomials $F_n(x)$ -- the derivative of $F_n(x)$ does not belong to the same family (as in the case of the classical orthogonal polynomials).  However, the roots of $F'_n(x)$ are the nontrivial $n$-th roots of unity, hence are independently equipped with a good deal of algebraic and arithmetic symmetry.  Thus, from the point of view of discriminants, our family $\lbrace F_n(x) \rbrace$ is more amenable to computation than a ``random'' family of polynomials, but is harder to work with than a family of classical orthogonal polynomials.  This manifests itself in (\ref{disc_form}) where we understand some of the prime factorization quite well (viz., $n/L_n^{n-1}$) and the rest ($\mathcal{P}_n$) not at all.

\section{Preliminary Results}

We start with a concise proof that $\disc (F_n) \not \in \Q^{\times 2}$ when $n \not \equiv 1 \pmod{4}$.

\begin{prop} \label{easyprop}
If $n \equiv 0,2,\text{or } 3 \pmod{4}$, then $\disc (F_n) \not \in \Q^{\times 2}$.
\end{prop}

\begin{proof}
If $n \equiv 2,3 \pmod{4}$ then $\disc(F_n) <0$, by (\ref{disc}). For $n \equiv 0\pmod{4}$, we first compute $\disc (F_4) = 725/432 \not \in \Q^{\times 2}$.  If $n \geq 8$, then there exists a prime number in the interval $(n/2, n-2)$.   Fix such a prime $\ell$ and observe that $v_\ell(n) = 0$ and that  $v_{\ell}(L_n) = 1$. Thus $v_\ell \left(n/L_n^{n-1} \right)$ is odd.  For a nontrivial $n$-th root of unity $\theta$, consider $\widetilde{F}_n(\theta)$.  Reducing (\ref{simplified}) modulo $\ell$, we have $\widetilde{F}_n(\theta) \equiv \left(L_n/\ell\right)\theta^\ell \pmod{\ell}$, whence $\mathcal{P}_n$ is coprime to $\ell$:
\[
\mathcal{P}_n =  \prod_{\lbrace \theta : F_n'(\theta)=0 \rbrace} \widetilde{F}_n(\theta) \equiv \prod_{\lbrace \theta : F_n'(\theta)=0 \rbrace} \frac{L_n}{\ell} \theta^\ell \pmod{\ell} \equiv \left( \frac{L_n}{\ell}\right)^{n-1}\cdot (-1)  \not \equiv 0 \pmod{\ell}.
\] 
Thus $v_\ell(\disc (F_n))$ is odd and so $\disc (F_n) \not \in \Q^{\times 2}$.
\end{proof}

This brings us to the case $n \equiv 1 \pmod{4}$, which is the most mysterious.  We start by generalizing \cite[Thm.~3.1]{ijnt} to when $n$ is a prime power.

\begin{thm} \label{prime_power}
Let $p \equiv 1 \pmod{4}$, $e$ a positive integer, and $n=p^e$.  Then $\mathcal{P}_n$ is coprime to $p$.
\end{thm}

\begin{proof}
We start by reducing the  factors $\widetilde{F}_n(\theta)$ of $\mathcal{P}_n$ modulo $p$:
\begin{align*} 
a_0 &= L_n + \frac{L_n}{n} - \frac{L_n}{n-1} \equiv  \frac{L_n}{n} \pmod{p} \not \equiv 0 \pmod{p}, \text{ and} \\
a_k  &= \frac{L_n}{k} - \frac{L_n}{n-1} \equiv 0 \pmod{p}
\end{align*}
for all $k = 1,\dots n-2$, due to the fact that $0 \leq v_p(k) \leq e-1$ for all such $k$.  Therefore 
\[
\mathcal{P}_n  = \prod_{\lbrace \theta : F_n'(\theta)=0 \rbrace} \widetilde{F}_n(\theta) \equiv \prod_{\lbrace \theta : F_n'(\theta)=0 \rbrace} \frac{L_n}{n} \equiv \left( \frac{L_n}{n} \right)^{n-1} \not \equiv 0 \pmod{p},
\]
as claimed.
\end{proof}

\begin{cor}
Let $p \equiv 1 \pmod{4}$, $e$ a positive integer, and $n=p^e$.  If $e$ is odd, then $\disc(F_n) \not \in \Q^{\times 2}$.  If $e$ is even, then  $\disc(F_n) \not \in \Q^{\times 2}$ if and only if  $\mathcal{P}_n \not \in \Q^{\times 2}$.
\end{cor}

\begin{proof}
Note that whether $e$ is even or odd, we have 
\[
v_p \left(\disc(F_n) \right) = v_p(n) - (n-1)v_{p}(L_n) + v_p(\mathcal{P}_n) = e - e(n-1) +0,
\]
by Theorem \ref{prime_power}.  If $e$ is odd, then $v_p\left(\disc(F_n) \right)$ is odd (because $n \equiv 1 \pmod{4}$) and hence $\disc(F_n) \not \in \Q^{\times 2}$.  If $e$ is even, then $n/L_n^{n-1}$ is a rational square, hence $\disc(F_n)$ is a rational square if and only if $\mathcal{P}_n$ is.
\end{proof}

By Example \ref{n=9}, we cannot expect to make much general progress on the case when $n \equiv 1 \pmod{4}$ when $n$ is a square, due to the unpredictable prime factorization of $\mathcal{P}_n$.    For the rest of the paper, we assume $n \equiv 1 \pmod{4}$ is not a square.

\section{The Case $n=mq$} 

Fix a positive integer $m$, let $\omega$ be a primitive $m$th root of unity, and define 
\begin{align*}
X(m) &\ddef \prod_{k=1}^{m-1}  \frac{1}{m} + \omega^k + \frac{1}{2}\omega^{2k} + \cdots + \frac{1}{(m-1)}\omega^{(m-1)k}  \\
Y(m) &\ddef 1 + \frac{1}{2} + \cdots + \frac{1}{m}.
\end{align*}
Observe that both $X(m)$ and $Y(m)$ are rational numbers. Consider the set 
$$
{\rm E}_m = \{\text{primes $\ell$}\mid\text{$v_\ell(X(m))> 0$, $v_\ell(Y(m))> 0$, and $m \ell \equiv 1 \pmod{4}$}\}
$$
consisting of all prime divisors of $X(m)$ and $Y(m)$ whose product with $m$ is congruent to $1\pmod{4}$.

\begin{remark} While ${\rm E}_m$ can be computed explicitly for fixed $m$, we expect that saying anything in general about this set could be difficult. Even without considering $X(m)$, the set ${\rm E}_m$ relies on understanding $p$-adic properties of the sequence $Y(m)$ of harmonic numbers, which is known to be a hard problem (see \cite{boyd}). 
\end{remark}

\begin{thm} \label{mainthm2restated}
For all primes $q$ such that $q \not \in {\rm E}_m$, $\gcd(q,m) = 1$, and $n \ddef mq \equiv 1 \pmod{4}$,
we have $\disc(F_n) \not \in \Q^{\times 2}$.
\end{thm}

\begin{proof}
Recall that 
\begin{equation}\label{discFmq}
\disc(F_{mq}) = \frac{mq}{L_{mq}^{mq-1}} \mathcal{P}_{mq}.
\end{equation}
Since $mq \equiv 1 \pmod{4}$, we have that $v_q\left( \frac{mq}{L_{mq}^{mq-1}} \right)$ is odd.  Thus, if $v_q(\mathcal{P}_{mq})$ is even, then $\disc(F_{mq}) \not \in \Q^{\times 2}$.  We now compute $\mathcal{P}_{mq} \pmod{q}$.

Let $\theta_k = \exp(2 \pi i/mq)^k$ and $\omega = \exp(2 \pi i/m)$.  Then 
\[
\mathcal{P}_{mq}  = \prod_{k=1}^{mq-1} \left( \sum_{j=0}^{mq-2} a_j\theta_k^j \right).
\]
Reducing each coefficient modulo $q$ gives us
\begin{align*}
\mathcal{P}_{mq}  &\equiv  \prod_{k=1}^{mq-1}  \left(a_0 + a_q\theta_k^q + a_{2q}\theta_k^{2q} + \cdots + a_{(m-1)q}\theta_k^{(m-1)q} \right) \pmod{q} \\
&\equiv \prod_{k=1}^{mq-1}  \left(a_0 + a_q\omega^k + a_{2q}\omega^{2k} + \cdots + a_{(m-1)q}\omega^{k(m-1)} \right) \pmod{q} \\
&\equiv \prod_{k=1}^{mq-1}  \left(\frac{L_{mq}}{mq} + \frac{L_{mq}}{q}\omega^k + \frac{L_{mq}}{2q}\omega^{2k} + \cdots + \frac{L_{mq}}{(m-1)q}\omega^{k(m-1)} \right) \pmod{q} \\
&\equiv \underbrace{\left(\frac{L_{mq}}{q}\right)^{mq-1}}_{\not \equiv 0 \pmod{q}} \prod_{k=1}^{mq-1}  \left(\frac{1}{m} + \frac{\omega^k}{1}+ \frac{\omega^{2k}}{2} + \cdots + \frac{\omega^{k(m-1)}}{(m-1)} \right) \pmod{q}.
\end{align*}

As $k$ ranges over $1,\dots,mq-1$, the product in the last congruence above 
can be rewritten as 
\begin{align*}
\prod_{k=1}^{mq-1}  \left(\frac{1}{m} + \frac{\omega^k}{1}+ \frac{\omega^{2k}}{2} + \cdots + \frac{\omega^{k(m-1)}}{(m-1)} \right) &= \left(X(m)Y(m)\right)^{q-1}X(m) \\&= X(m)^qY(m)^{q-1}.
\end{align*}
Thus
\begin{align*}
\mathcal{P}_{mq}&\equiv \left(\frac{L_{mq}}{q}\right)^{mq-1} X(m)^qY(m)^{q-1} \pmod{q},
\end{align*}
which is coprime to $q$ by the hypothesis that  $X(m)$ and $Y(m)$ are both coprime to $q$.  Thus $v_q(\mathcal{P}_{mq}) =0$ and so $\disc(F_{mq}) \not \in \Q^{\times 2}$. 
\end{proof}

\section{Examples and Conclusions}

We conclude the paper with several examples and observations.  We start by setting $m=p$ a prime number in Theorem \ref{mainthm2restated}.  If we fix $p$, then as long as both $X(p)$ and $Y(p)$ are coprime to $q >p$,  we can conclude that $\disc(F_{pq}) \not \in \Q^{\times 2}$.  For the remaining values of $q \in {\rm E}_p$, we can check by hand as long as it is computationally feasible.  The following result is a sample implementation of our methods.  We note that other examples with $n=pq$ are similarly easy to generate. 

\begin{prop} \label{pqprop}
For any prime $q$, if $n=3q$, $5q$, or $7q$, then $\disc(F_n) \not \in \Q^{\times 2}$. 
\end{prop}

\begin{proof}
If $n \equiv 2,3 \pmod{4}$, we are done by Proposition \ref{easyprop}.  We have seen in Example \ref{n=9} that $\disc(F_9) \not \in \Q^{\times 2}$.  By using the \texttt{issquare} command in \texttt{PariGP} we check that neither $\disc(F_{25})$ nor $\disc(F_{49})$ is a rational square.  

It therefore suffices to consider $n=3q$, $5q$, or $7q$ in the case $n \equiv 1 \pmod{4}$ and $q >3, 5, 7$, respectively.  By Theorem \ref{mainthm2restated}, we are reduced to checking finitely many cases.  In the following table we compute, for each value of $p$, the rational numbers $X(p)$ and $Y(p)$ and determine the set ${\rm E}_p$. We compute

\begin{center}
\begin{tabular}{llll}
$p$ & $X(p)$ & $Y(p)$ & ${\rm E}_p$ \\
\hline
3 & 13/36 &11/6& $\lbrace 11 \rbrace$ \\
5 & $(11\cdot 101 \cdot 3001)/(2^83^45^4)$&$137/(2^2\cdot 3 \cdot 5)$& $\lbrace 101,137,3001 \rbrace$\\
7 & $1170728665999621/(2^{12}3^65^67^6)$&$(3\cdot 11^2)/(2^2 \cdot 5 \cdot 7)$& $\lbrace 11 \rbrace$
\end{tabular}
\end{center}
In each case we find an auxiliary prime $\ell$ for which $\disc(F_n)$ is not a square modulo $\ell$, \emph{viz.}
\begin{align*}
\disc(F_{33}) &\equiv 14 \pmod{37} \\ 
\disc(F_{505}) &\equiv 200 \pmod{509} \\
\disc(F_{685}) &\equiv 443 \pmod{709} \\
\disc(F_{15005}) &\equiv 13652  \pmod{15017} \\
\disc(F_{77}) &\equiv 39 \pmod{79};
\end{align*}
none of these are squares.
\end{proof}

We close with several observations which suggest that a different approach than that discussed in this paper may be necessary in order to prove Yamamura's conjecture in general.  In particular, the condition that $n = mq$ with $q > m$ and $\gcd(m,q) = 1$ is a strong hypothesis that we cannot remove.  What cases are left to prove?  Among all $n \equiv 1 \pmod{4}$, either
\begin{itemize}
\item $n$ is an odd square, or
\item $n$ has at least two distinct prime divisors, and at least one prime divisor at which $n$ has odd valuation.
\end{itemize}
In the second case, Theorem \ref{mainthm2restated} handles values of the form
\[
n = q_1^{e_1} q_2^{e_2} \cdots q_{r-1}^{e_r-1}q_r,
\]
for prime numbers $q_1 < q_2 < \cdots q_r$.  With regard to generalizing the proof of Theorem \ref{mainthm2restated}, we present the following data as evidence that a different approach is needed.

\bigskip 

\begin{enumerate}

\item The condition $q > m$ in Theorem \ref{mainthm2restated} allows us to compute $\widetilde{F}_{mq}(\theta) \pmod{q}$ easily -- the only $a_k$ which are not divisible by $q$ are the $a_{jq}$ for $j=1,\dots,m-1$.  If $q<m$, then we would need to include additional multiples of $q$ less than $m$. 

\medskip

\noindent For example, consider $n=21 = 3\cdot7$.  By Proposition \ref{pqprop}, we know that $\disc(F_{21}) \not \in \Q^{\times 2}$ by working modulo 7.  However, if we were to localize at the prime $3$ instead, we first directly compute in \texttt{PariGP} that $v_{3}(\disc(F_{21})) = -34$, and then separately compute using \eqref{discFmq} that
\[
v_3 \left(\disc(F_{21}) \right) = v_3(21) - 20 v_3(L_{21}) + v_3( \mathcal{P}_{21}) = 1 - 40 + v_{3}(\mathcal{P}_{21}),
\]
whence $v_{3}(\mathcal{P}_{21}) = 5$.  Therefore, it is not the case that if $v_{p}(n)$ is odd then $v_{p} \left(\disc(F_{n})\right)$ is odd as well.  

\medskip

\noindent To finish this example, we note that the prime factorization of $\mathcal{P}_{21}$ is
\begin{align*}
\mathcal{P}_{21} &=3^5 \cdot 31 \cdot 41^2 \cdot 335642497 \cdot  1236257387 \cdot 11513876767 \\
&\times 1381773062083 \cdot 3484835094151 \cdot     2204197718654031818404984907 \\
&\times 9004989137610212635527213226585626310173203221874790587323\\
&\hspace{.17in}6753813403920291816681
\end{align*}
This computation was carried out on \texttt{PariGP} in 38.5 minutes on a personal laptop. \\

\item Setting $m=8$ in Theorem \ref{mainthm2restated} implies $n \equiv 0 \pmod{4$}, hence $m=9$ is the smallest composite value of $m$ for which we can have $n \equiv 1 \pmod{4}$.  In that case we compute
\begin{align*}
X(9) &= \frac{37 \cdot 229 \cdot 367 \cdot 98481394090065580021}{2^{24}3^{16}5^87^8} \\
Y(9) &= \frac{7129}{2^3 \cdot 3^2 \cdot 5 \cdot 7}.
\end{align*}
Thus, if $q \not \in \lbrace 37, 229, 7129, 98481394090065580021 \rbrace$, we can immediately conclude that $\disc(F_{9q}) \not \in \Q^{\times 2}$ (note that $367 \equiv 3 \pmod{4}$).  However, for the remaining values of $q$ we run into some of the computational limits of this question.

\medskip

\noindent Let $q=37$.  One can verify in \texttt{PariGP} that  $v_{37}(\mathcal{P}_{333}) = 37$, which agrees with 
\[
\mathcal{P}_{mq} \equiv \left(\frac{L_{mq}}{q}\right)^{mq-1} X(m)^qY(m)^{q-1} \pmod{q}
\]
from the proof of Theorem \ref{mainthm2restated}. Therefore, none of the primes $<333$ can be used to immediately deduce whether or not $\disc(F_{333})$ is a rational square.  The next prime larger than $333$ is $337$ and we check that $\disc(F_{333}) \equiv 157 \pmod{337}$, which is not a square; this computation took 16.2 seconds on a personal computer.  However, a similar analysis is computationally infeasible for the remaining values of $q$. \\

\item It is not necessarily the case that if $p>n$, then $v_p(\disc(F_n))$ is odd (which would automatically imply $\disc(F_n) \not \in \Q^{\times 2}$).  For example, we have $v_{4019}(\disc(F_{15})) = v_{4019}(\mathcal{P}_{15})= 2$.  \\

\item We checked all odd square values of $n$ from 1 to 1000 and in none of those cases do we find that $\disc(F_{n})$ is a rational square. \\

\end{enumerate}

\end{document}